\documentclass[11pt,mathserif]{amsart}
\usepackage[utf8]{inputenc}
\usepackage{eulervm}
\usepackage{amsmath}
\usepackage{amsthm}
\usepackage{import}
\usepackage{geometry}
\usepackage{mathrsfs} 
\usepackage{amsfonts}
\usepackage{amssymb}
\usepackage{palatino,color}
\usepackage{hyperref}
\usepackage{float}
\usepackage[font=small,skip=-10pt]{caption}
\usepackage{comment,accents}
\usepackage[export]{adjustbox}
\usepackage{enumerate,graphicx,enumitem,enumerate}

\newtheorem{mydef}{Definition}[section]

\newtheorem{prop}{Proposition}[section]
\newtheorem{theorem}{Theorem}[section]
\newtheorem{lemma}{Lemma}[section]

\newtheorem{claim}{Claim}[section]

\newcommand{\spsp}{\mathcal{S}_3}
\newcommand{\De}{\mathrm{d}}

\newcommand{\bft}{\mathcal{T}}
\newcommand{\velmap}{\mathbb{V}}
\newcommand{\tsp}{\mathbf{T}}
\newcommand{\covdev}{\nabla^{W_2} }

\newcommand{\R}{\mathbb{R}}

\newcommand{\RD}{\mathbb{R}^d}

\newcommand{\cC}{\mathcal{C}}

\newcommand{\cP}{\mathcal{P}}

\newcommand{\bes}{\begin{equation*}}
\newcommand{\ees}{\end{equation*}}
\newcommand{\beas}{\begin{eqnarray*}}
\newcommand{\eeas}{\end{eqnarray*}}
\newcommand{\bea}{\begin{eqnarray}}
\newcommand{\eea}{\end{eqnarray}}
\newcommand{\be}{\begin{equation}}
\newcommand{\ee}{\end{equation}}

\newcommand{\bbX}{\mathbb{X}}

\newcommand{\bbl}{\begin{block}}
\newcommand{\ebl}{\end{block}}

\newcommand{\hP}{\hat{P}}
\newcommand{\hgamma}{\hat{\gamma}}
\newcommand{\hpi}{\hat{\pi}}
\newcommand{\bbV}{\mathbb{V}}
\title{Measure-valued spline curves: an optimal transport viewpoint}
\author{ Yongxin Chen, Giovanni Conforti and Tryphon Georgiou}
\begin{document}
	\maketitle
	
	\begin{abstract}
The aim of this article is to introduce and address the problem to smoothly interpolate (empirical) probability measures. To this end, we lift the concept of a spline curve from the setting of points in a Euclidean space that that of probability measures, using the framework of optimal transport. 
\end{abstract}
	
\section{Introduction}
Consider a collection of (empirical) probability distributions 
\[
(\rho_i)_{i=0,1,\ldots,N},
\]
that are specified at a number of successive points in time $0=t_0< t_1 \ldots < t_N=1$. From an engineering standpoint, such distributions may represent density of particles, concentration of pollutants, image intensity, power distribution, etc., associated with some underlying time-varying physical process. In pertinent application areas, invariably, the goal is to interpolate the available data-set so as, e.g., to estimate the spread of a particle beam or the potential spread of polutants in-between reference points, to resolve features between successive slices in magnetic resonance imaging, and so on. Thus, our aim is to construct in a systematic manner a measure-valued curve which interpolates \emph{smoothly} a data-set that consists of successive probability distributions, and to develop suitable computational tools for this purpose.

In a classical setting, where our data-set consists of points $(x_i)_{i=0,1,\ldots,N}$ in $\RD$, a natural choice is to interpolate with a smooth curve such as a cubic spline. This motivates us to seek a suitable generalization of spline curves from the Euclidean setting to {\em measure-valued spine curves} on the Wasserstein space of probability measures. We achieve this by adopting a variational formulation of splines due to Holladay \cite{holladay1957smoothest}, that the spline-curve in Euclidean space minimizes mean-squared acceleration among all other interpolants, to the setting of optimal transport theory.

Besides certain expected parallels to classical splines, measure-valued splines enjoy a number of interesting structural properties which mirror other well known properties of optimal transport. In particular, we show that the construction of measure-valued splines relates to a multimarginal optimal transportation problem (see \cite{gangbo1998optimal,pass2015multi}), and we discuss the existence of Monge-like solutions for an extended (relaxed) formulation of the multimarginal optimal transport problem. We also provide a heuristic fluid dynamic formulation for splines, which may be regarded as the counterpart to the Benamou-Brenier formulation of the Monge-Kantorovich problem. As an illustrative example, we expand on the case where the $\rho_i$'s are Gaussian measures. In this case the one-time marginal distributions are Gaussian at all times and the measure-valued splines can be explicitly computed by solving a semidefinite program. Lastly, based on the fact put forward by Otto \cite{otto2001geometry} that we may regard the Wasserstein space as an almost-Riemannian infinite dimensional manifold, we discuss an alternative approach to constructing measure-valued splines and provide a formal argument showing that the original optimization problem to define splines is in fact a relaxation of the one stemming from this Riemannian viewpoint.

The results of this article should be considered as a first step towards developing a toolbox for interpolation in the space of probability measures; some of the most basic elements of the theory are proven rigorously whereas only formal arguments are given for other claims. In light of the range of potential applications, besides resolving certain open questions that are raised, future work may need to focus on more general \emph{smoothing} splines, or B-splines, as well as on developing fast and efficient computational algorithms.

\subsubsection*{Organization of the paper}

In Section \ref{sec:formulation}, we define the notion of spines in Wasserstein space by emulating its well known Euclidean counterpart.
Section \ref{sec:measuresplines} explores the structure of such measure-valued splines and, in particular, points out that the measure (which is the sought matrix-valued spline) is concentrated
on ordinary $C^2$ spline curves. It also presents alternative formulations (e.g., in phase space) as well as discusses the question of Monge solutions.
Section \ref{sec:fluid} presents yet another formulation that is analogous to the Benamou-Brenier fluid dynamical formulation of standard Optimal Mass transport.
Section \ref{sec:riemannian} explores yet another angle of viewing measure-valued splines. It relies on Otto calculus on Wasserstein space and brings out the problem to minimize acceleration subject to constraints.
Section \ref{sec:proofs} contains proofs of the main results.
We conclude by specializing Wasserstein-spline interpolation to Gaussian data in Section \ref{sec:gaussian} and we highlight the typical outcome with examples that are presented in the final section, Section \ref{sec:numerical}.

\subsubsection*{Notation}
We introduce here notation which we use throughout the paper. For $k \geq 0$ and integer, we denote the set of functions $X:[0,1] \rightarrow \RD$ which are continuous and $k$ times continuously differentiable by $ C^k([0,1];\RD)$ and abbreviate by $C^k$. The set of functions which are $k$ times differentiable and whose $k$-th derivative is square-integrable we denote by  $H^k([0,1];\RD)$, abbreviated by $H^k$. Splines are, by definition, twice continuously differentiable and piecewise cubic polynomials.
 Thus, for a fixed sequence $\mathcal{T}:=(t_i)_{i=0,\ldots,N}$ with $0=t_0<t_1<\ldots<t_N=1$ 
 we denote by $\Pi_3([t_{i},t_{i+1}])$ the set of  $\RD$-valued cubic polynomials defined on the interval $[t_{i},t_{i+1}]$ and the corresponding set of splines 
\bes \spsp := \Big\{ X \in C^2([0,1];\RD) :  X \big |_{[t_{i},t_{i+1}]} \in \Pi_3([t_{i},t_{i+1}]) \quad \forall i=0,\ldots, N-1 \Big\}.
\ees

We denote by $\mathcal{P}(\Omega)$ the space of probability measures over a measurable space $\Omega$ and by $\cP_2(\RD)$ the subset of the elements of $\cP(\RD)$ having finite second moment. We will often choose $\Omega=C^0$, which we equip with the canonical sigma algebra generated by the projection maps $(X_t)_{t \in [0,1]}$, defined by 
 \[ \forall \omega \in \Omega, \quad X_t(\omega)=\omega_t. \]
 If $\mathcal{T}=(t_i)_{i=0,\ldots,N}$ is a finite set of times, we denote $X_{\bft}$ the vector $(X_{t_0},X_{t_1},\ldots, X_{t_N})$. 
Finally, if $T$ is a map and $\mu$ a probability measure, we denote $T_{\#}\mu$ the push forward of $\mu$ under $T$.
 
\section{Problem formulation}\label{sec:formulation}

We now draw on the analogy between curve fitting in finite-dimensions and interpolation in the Wasserstein space to define our problem of constructing smooth trajectories (splines) in the Wasserstein space.

\subsection{ Natural interpolating splines in $\R^d$}
Let $\mathcal{T}:=(t_i)_{i=0,\ldots,N}$ with $0=t_0<t_1<\ldots<t_N=1$ be an array of time-data, and $(x_i)_{i=0,\ldots,N}$ be a sequence of spatial data in $\RD$. The natural interpolating spline for the data is the only $S\in\spsp$ such that $S_{t_i} = x_i$ for $0 \leq i \leq N$ and whose second derivative vanishes at $t=0,1$. Holladay's Theorem \cite{holladay1957smoothest} tells that the variational problem
\begin{subequations}\label{eq:spline}
\begin{align}
 &{}&\inf_X \int_0^1 |\ddot{X}_t|^2 dt\\
&{}&X \in H^2, \\
&&\ X_{t_i} = x_i,  \quad i=0,\ldots,N.
\end{align}
\end{subequations}
 admits as unique solution  the natural interpolating spline for the data $(t_i,x_i)_{i=0,\ldots,N}$, which we denote
 $S(x_0,\ldots, x_N)$.  We do not emphasize the dependence on the time data $\mathcal{T}$, as they are kept fixed throughout the article.
 Also, we denote $\spsp^0\subset\spsp$ the set of all natural splines
 \[\spsp^0= \{ S(x_0,\ldots, x_N) : (x_0,\ldots,x_N) \subseteq \R^{d\times (N+1)} \}. \]

\subsection{Interpolating splines in $\cP_2(\RD)$}\label{sec:formulation}

Starting from the given data $(t_i,\rho_i)_{i=0,..,N}$, with $0=t_0< t_1 \ldots <t_N=1$ and $ \{\rho_{0}, \ldots \rho_{N} \} \subseteq \cP_2(\RD)$,  inspired by Holladay's theorem and with an optimal transport viewpoint, we view the problem of interpolating smoothly the data as
\begin{center}
`` the problem of transporting the mass configuration $\rho_{0}$ into the mass\\configuration $\rho_{i}$ at time $t_i$ while minimizing mean-squared acceleration."
\end{center}
To propose a model, we make the following observations motivated by the above informal description of our problem.
\begin{itemize}
\item We view a transport plan as a probability measure $P\in \cP(\Omega)$, where $\Omega=C^0$ and  for $A \subseteq \Omega$, $P(A)$ represents the total mass which flows along the paths in $A$.
\item For a plan to be admissible, it must be that at time $t_i$, the mass configuration induced by $P$ is $\rho_i$. Thus, we ask that
\[ (X_{t_i})_{\#}P= \rho_i, \quad i =0,\ldots,N. \]
\item Since we consider acceleration (of a curve in Wasserstein space), we ask that an admissible plan $P$  is such that $P(H^2)=1$.
\item Since we penalize acceleration, we need to consider the mean-square acceleration\footnote{When $(X_t)_{t\in [0,1]}$ is the canonical process, we denote the acceleration $\partial_{tt}X_t$ instead of $\ddot{X}_t$.} 
\begin{equation}\label{eq:Kantorovich_cost}
 \int_0^1\int_{\Omega}|\partial_{tt}X_t|^2 \, \De P  \De t
\end{equation}
of an (admissible) plan $P$.
\end{itemize}

We are now in the position to define measure-valued spline curves.
\begin{mydef}
Let $(t_i,\rho_i)_{i=0,\ldots,N} \subset [0,1]\times \cP_2(\RD)$ be given data.  Consider the problem
\begin{subequations}\label{MSP}
\begin{align}
&{}&\label{MSPa} \inf_{P}   \int_0^1\int_{\Omega} |\partial_{tt}X_t|^2 \De P \, \De t \\
&&\label{MSPb} P \in \cP(\Omega),P(H^2)=1\\
 &&\label{MSPc} (X_{t_i})_{\#}P = \rho_i, \quad i=0,\ldots,N. 
\end{align}
An {\em interpolating spline } for the data $(t_i,\rho_i)_{i=0,\ldots,N}$ is defined to be the marginal flow $(\rho_t)$ of an optimal measure for \eqref{MSP}.
\end{subequations}
\end{mydef}
We remark that, if instead of taking the second derivative in \eqref{MSPa} we take the first derivative, then problem \eqref{MSP} is an equivalent formulation of Monge-Kantorovich problems within each time interval $[t_i,t_{i+1}]$.
Also we note that, in general, we cannot guarantee uniqueness for the optimal measure in \eqref{MSP}. Thus, the above definition may not define a \emph{natural} interpolating spline without additional hypothesis on the data (so as to ensure uniqueness).

\subsection{Compatibility}

As a first result we have that the definition we gave is compatible with that of splines in $\RD$.  
\begin{prop}\label{prop:compatibility}
Let$(t_i,x_i)_{i=0.\ldots,N} \subset [0,1] \times \RD$, and set $\rho_{i}: = \delta_{x_i}$ for $0 \leq i \leq N$. Then the unique optimal solution of \eqref{MSP} is 	\[ P^* = \delta_{S}, \]
where $S$ is the natural interpolating spline for $(t_i,x_i)_{i=0,\ldots,N}$.
\end{prop}
We shall see that the above proposition is a special case of Theorem \ref{thm:splinestruc} below.

\section{The structure of measure-valued splines}\label{sec:measuresplines}
 
\subsection{Decomposition of optimal solutions}
 The following theorem asserts that at least an optimal solution for \eqref{MSP} exists and gives details about the structure of the solution. In the present article, we do not establish uniqueness of the measure-valued spline through a given data set; this interesting question remains open for further investigation. In words, Theorem \ref{thm:splinestruc} says that any optimal solution is supported on splines of $\RD$, and that its joint distribution at times $t_0,\ldots,t_N$ solves a multimarginal optimal transport problem whose cost function $\mathcal{C}$ is the optimal value in \eqref{eq:spline}, i.e.
\be\label{eq:multimargcost}
\mathcal{C}(x_0,\ldots,x_N) := \int_{0}^1 |\partial_{tt}S_t(x_0,\ldots,x_N)|^2\De t.
\ee
Thus a spline curve on $\cP_2(\RD)$ is found by pushing forward through splines of $\RD$ the solution of a multimarginal optimal problem. This is in analogy with the well known fact that the geodesics of $\cP_2(\RD)$ are constructed pushing forward the optimal coupling of the Monge-Kantorovich problem through geodesics of $\RD$\,(\cite[Theorem 2.10]{ambrosio2013user}).  In the statement of the theorem, as usual, we set 
\[ \Pi(\rho_0,\rho_1,\ldots,\rho_N) = \left\{ \pi \in \cP(\RD \times \ldots  \times \RD) : (X_{i})_{\#} \pi = \rho_i \right\}, \]
where we denoted by $X_i$ the $i$-th coordinate map on $(\RD)^{ N+1}$, i.e.  $X_i(x_0,\ldots,x_N) =x_i$.

\begin{theorem}\label{thm:splinestruc}
Let $\{\rho_0,\ldots,\rho_N \}\subseteq \cP_2(\RD)$. Then there exists at least an optimal solution for \eqref{MSP}. Moreover, the following are equivalent
\begin{enumerate}[label= (\roman*),ref= (\roman*)]
\item\label{item(i)} $\hat{P}$ is an optimal solution for \eqref{MSP}.
\item\label{item(ii)}  $\hat{P}(\spsp^0)=1$ and $\hat{\pi}:= (X_{\bft})_{\#} \hat{P}$ is an optimal solution for 
\bea\label{eq:mulitmargprob}
&{}& \inf_{\pi}\int \mathcal{C}(x_0,x_1,\ldots,x_N) \De \pi \\
\nonumber &{}& \pi \in \Pi(\rho_0,\rho_1,\ldots,\rho_N), 
\eea
where $\mathcal{C}$ has been defined at \eqref{eq:multimargcost}.
\end{enumerate}
\end{theorem}

Multimarginal optimal transport problems, such as the one in \eqref{eq:mulitmargprob}, can be solved numerically using iterative Bregman projections \cite{BenCar15}. However, this approach is computational burdensome for high dimensional distributions or large number of marginals. In the special case where the marginals are Gaussian distributions, a numerically efficient semidefinite programming (SDP) formulation is possible (see Section \ref{sec:gaussian}). 

\subsection{Formulation of the problem in phase space}\label{sec:phase_space}
One aspect of the cost $\mathcal{C}$ which complicates the tractability of \eqref{eq:mulitmargprob} is that, to the best of our knowledge,  there is no closed form expression valid for any $N$. For this reason, we propose a second, equivalent formulation of \eqref{eq:mulitmargprob} in a larger space with an explicit cost function. Note that a very simple reformulation of \eqref{MSP} can be obtained by looking into ``phase space". Here we consider probability measures on the product space $ H^1 \times H^1$, where we define canonical projection maps $(X_t)_{t \in [0,1]}$ and $(V_t)_{t \in [0,1]}$ in the obvious way. The problem 

\begin{subequations}\label{pathphasesp}
\begin{align}
\inf_{Q}\int_{0}^1 \int_{\Omega \times \Omega} | \partial_t V_t |^2 \De Q \De t\\
Q \in \cP\left(   \Omega \times \Omega  \right), Q(H^1 \times H^1)=1, \\
 Q(\partial_t X_t = V_t \ \forall t \in [0,1] )=1,\\
  \ (X_{t_i})_{\#}Q = \rho_i \quad  i =0, \ldots,N,
\end{align}
\end{subequations}
is easily seen to be equivalent to \eqref{MSP}. The interesting fact is that, the multimarginal optimal transport problem associated with \eqref{pathphasesp} has an explicit cost function.
 All relies on the following representation of $\mathcal{C}$ as the solution of a minimization problem.
\begin{lemma}\label{lm:costref}
Let $(x_i,v_i)_{i=0,\ldots,N} \subset \RD \times \RD$ be given. The optimal value of the problem
\begin{subequations}\label{eq:spline+vel}
\begin{align}
&{}&\inf_{X,V} \int_0^1 |\dot{V}_t|^2 dt\\
&{}&(X,V) \in H^1 \times  H^1,  \\
&& \dot{X}_t=V_t, \quad \forall t \in [0,1],\\
&& X_{t_i} = x_i, \quad i=0,\ldots , N, \\
&& V_{t_i} = v_i, \quad i=0,\ldots, N.
\end{align}
\end{subequations}
is given by
\be\label{eq:explicitcost}
\sum_{i=0}^{N-1} (t_{i+1}-t_i)^{-1} c( x_i,x_{i+1},v_i,v_{i+1} )
\ee
where 
\be\label{eq:explicitcost_formula}
c(x_i,x_{i+1},v_i,v_{i+1})= 12|x_{i+1} - x_{i} -v_i|^2 - 12 \langle x_{i+1} - x_{i} -v_i,\, v_{i+1}-v_{i} \rangle + 4 | v_{i+1}-v_{i}|^2.
\ee
 In particular,
\be\label{eq:costref1}
\mathcal{C}(x_0,\ldots,x_N) = \inf_{v_0, \ldots v_N \in \RD} \sum_{i=0}^{N-1} (t_{i+1}-t_i)^{-1} c( x_i,x_{i+1},v_i,v_{i+1} )
\ee
and the infimum in \eqref{eq:costref1} is attained and is unique.
\end{lemma}

We note that multimarginal optimal transport problems have been studied in \cite{carlier2010matching}
but for a cost of the form
\[  \mathcal{C}(x_0,\ldots,x_N) = \inf_{y \in Y} \sum_{i=0}^N c_i(x_i,y). \]
The main difference with the above is that $c$  in \eqref{eq:costref1} depends on both $x_i$ and $x_{i+1}$, which somewhat complicates the analysis; more details on this can be found in Section \ref{sec:gaussian}.

\begin{theorem}\label{thm:phasesplinestruc}
Let $\{\rho_0,\ldots,\rho_N \}\subseteq \cP_2(\RD)$. Then there exists at least an optimal solution for \eqref{pathphasesp}. Moreover, for an admissible plan $\hat{Q}$ the following are equivalent
\begin{enumerate}[label= (\roman*),ref= (\roman*)]
\item\label{item(i)phasestr} $\hat{Q}$ is an optimal solution for \eqref{pathphasesp}
\item\label{item(ii)phasestr}  $\hat{Q}(X \in \spsp^0)=1$ and $\hat{\gamma}:= (X_{\bft},V_{\bft})_{\#} \hat{Q}$ is an optimal solution for 
\begin{eqnarray}\label{eq:multimargphasespace}
c_{\rm opt}:=\inf_{\gamma} \sum_{i=0}^{N-1} (t_{i+1}-t_{i})^{-1} \int c(x_{i},x_{i+1},v_{i},v_{i+1} ) \De \gamma \\
\nonumber \gamma \in \Gamma( \rho_0, \ldots , \rho_N )
\end{eqnarray}
where $\Gamma(\rho_0,\ldots,\rho_N) $ is defined by
\begin{equation*}
\Gamma(\rho_0,\ldots,\rho_N):= \left\{ \gamma \in \mathcal{P}(\R^{d \times (N+1)} \times \R^{d \times (N+1)})  \,: (X_i)_{\#}\gamma = \rho_i \ \forall i =0,\ldots,N \right\}.
\end{equation*}
\end{enumerate}
\end{theorem}

 In the next proposition we show equivalence between the two multimarginal problems. There, we denote $\mathbb{V}$ the maps that associates to $(x_0,\ldots,x_N)$ the optimal solution of \eqref{eq:costref1} .  It is not hard to see that $\mathbb{V}$ is a linear map.

\begin{prop}\label{lem:pahsespaceform}
The problem \eqref{eq:multimargphasespace} is equivalent to the problem \eqref{eq:mulitmargprob} in the following sense:
\begin{enumerate}[label= (\roman*),ref= (\roman*)]
\item\label{item(i)phase} If $\hat{\gamma}$ is optimal for \eqref{eq:multimargphasespace} then 
\[ \hat{\pi}:= (X_0,\ldots,X_N)_{\#} \hat{\gamma}\]
 is optimal for \eqref{eq:mulitmargprob}. 
 \item\label{item(ii)phase} If $\hat{\pi}$ is optimal for \eqref{eq:mulitmargprob}, then 
 \[\hat{\gamma}:= (X_0,\ldots,X_N,\velmap(X_0,\ldots, X_N) )_{\#}\hat{\pi} \]
  is optimal for \eqref{eq:multimargphasespace}. 
\end{enumerate}
\end{prop}

\subsection{Monge solutions} Here, we discuss Monge, or graphical, solutions to the extended formulation. Unfortunately, we cannot provide a complete existence result. However, we show that if an optimal solution has some regularity properties, then it is of Monge form.

\begin{theorem}\label{thm:MongeSol}

Let $\hgamma$ be an optimal solution for \eqref{eq:multimargphasespace} such that for all $i=0,\ldots,N-1$ the measure $\hgamma_{i} \in \cP(\RD \times \RD)$ defined by
\[ \hgamma_i = (X_i,V_i)_{\#} \hat\gamma \]
is absolutely continuous w.r.t to the Lebesgue measure. Then there exist a map
\[  \Phi =(\varphi_1,\ldots,\varphi_N,\psi_1,\ldots,\psi_N): \RD \times \RD \rightarrow \R^{d \times N} \times \R^{d \times N}  \]
such that $\hgamma$ is concentrated on the graph of $\Phi$, i.e.
\[  \hgamma = (\mathbf{id}_{\R^{d}}, \varphi_1,\ldots,\varphi_N, \mathbf{id}_{\RD}, \psi_1,\ldots,\psi_N)_{\#} \gamma_0,  \]
or equivalently
\begin{equation}\label{eq:MongeSol1}
\hgamma\left( \bigcap_{i=1}^N \left\{ X_i = \varphi_i(X_0,V_0) \, ,V_i=\psi_i(X_0,V_0) \right\} \right) =1.
\end{equation}

\end{theorem}

It would be very desirable to derive the conclusion assuming just regularity of the $(\rho_i)$ instead of the $\hgamma_i$. Theorem \ref{thm:MongeSol} implicitly tells that Monge solutions for \eqref{eq:mulitmargprob} are not to be expected; the support of an optimal solution should be locally of dimension $2d$. We also believe that the assumptions of the Theorem can be largely relaxed. In the next propostion we take a first step in this direction for the case when  $N=2$ (i.e. we interpolate three measures), using the general results of \cite{pass2012structure}.

\begin{prop}\label{prop:threepoints}
Let $N=2$, $\hpi$ an optimal solution for \eqref{eq:mulitmargprob}, and $(x_0,x_1,x_2)$ a point in the support of $\hpi$. Then there is a neighborhood $O$ of $(x_0,x_1,x_2)$ such that the intersection of the support of $\hpi$ with $O$ is contained in a Lipschitz submanifold of dimension $2d$.
\end{prop}
Let us note that this proposition does not yield the existence of Monge solutions for \eqref{eq:multimargphasespace}; however it proves that optimal solutions of \eqref{eq:mulitmargprob} have a support which is locally of dimension $2d$, without making any further regularity assumption on the optimal coupling.
\section{Fluid dynamical formulation of \eqref{MSP}}\label{sec:fluid}
To better understand what follows, lets us recall the fluid dynamic formulation of the Monge-Kantorovich problem, which is due to Benamou and Brenier. In  \cite{benamou2000computational}  they showed that the optimal value for
\begin{subequations}
\begin{align*}
 &{}&\label{BB}\tag{BB} \inf_{\mu,v} \int_{0}^1 \int_{\RD}|v_t|^2(x) \mu_t(x) \De x  dt\\
&{}& \nonumber\partial_t \mu_t(x)+ \nabla \cdot (v_t \mu_t)(x)=0 \\
&{}& \nonumber \mu_0=\rho_0,\mu_1=\rho_1
\end{align*}
\end{subequations}
is the squared Wasserstein distance $W^2_2(\rho_0,\rho_1)$ and that the optimal curve is the displacement interpolation \cite{mccann1997convexity}.

  \subsection{A fluid dynamic formulation for \eqref{MSP}}

Inspired by   \eqref{eq:multimargphasespace}, we formulate the following problem

\begin{subequations}\label{FDF}
\begin{align}
&{}&\label{FDFa} \inf_{\mu,a} \int_{0}^1 \int_{\RD} |a_t(x,v)|^2  \mu_t(x,v) \De x \De v  \\
&{}&\label{FDFb}   \partial_{t} \mu_t(x,v)+ 
\langle \nabla_x \mu_t(x,v), v \rangle 
+ \nabla_v \cdot (a_t\mu_t)(x,v)=0, \\
&{}&\label{FDFc} \int_{\RD} \mu_{t_i}(x,v) dv = \rho_{t_i}, \quad i=0,\ldots,N,
\end{align}
\end{subequations}
where we denote by $\nabla_x$(resp. $\nabla_v$) the gradient taken w.r.t. the $x$ (resp. $v$) variables, so that $\nabla_x \cdot$ stands for the divergence taken w.r.t. the $x$ variables, and similarly for $\nabla_v \cdot$.
\begin{claim}\label{claim:equivalence}
The two problems \eqref{MSP} and \eqref{FDF} are equivalent.
\end{claim}
We provide formal calculations to justify the claim. However, the argument below does not constitute a rigorous proof as it rests on assuming existence of Monge-like solutions for \eqref{eq:multimargphasespace}, which we only proved under certain assumptions. Moreover, we take derivatives formally and do not insist here on justifying their existence , and in which sense we should consider them.
\begin{proof}[Sketch of the argument]
Let $(\mu,a)$ be an optimal solution for \eqref{FDF}. The constraint \eqref{FDFb} implies that the vector field $(w_t)$ solves the continuity equation for $(\mu_t)$, where
\be\label{eq:wt} w_t = \begin{pmatrix} v \\ a_t(x,v) \end{pmatrix} \ee
Thus, if we consider the flow maps $(\bbX_t, \bbV_t)_{t \in [0,1]}$ for $w_t$, defined by
\begin{equation}\label{eq:phasespaceflowmaps}
\partial_t  \begin{pmatrix}
 \mathbb{X}_t\\ 
\mathbb{V}_t
\end{pmatrix}
=
\begin{pmatrix}
\mathbb{V}_t\\
 a_t(\mathbb{X}_t \mathbb{V}_t)
\end{pmatrix}, \quad  \begin{pmatrix}
 \mathbb{X}_0 \\ 
\mathbb{V}_0 
\end{pmatrix} = \ \mathbf{id}_{\RD \times \RD}
\end{equation}
then we have that 
\begin{equation}\label{eq:phasespacepushforward}
 \forall t \in [0,1], \quad (\bbX_t,\bbV_t )_{\#} \mu_0 = \mu_t .
\end{equation}
In particular, because of \eqref{FDFc}
\be\label{eq:admissibility} \forall i=0,\ldots,N \quad (\mathbb{X}_{t_i})_{\#}\mu_0= \rho_i.
\ee
Define $P \in \cP(\Omega)$ as follows
\be\label{eq:phasespaceMongesol}  P:=  ((\mathbb{X}_t)_{t \in [0,1]})_{\#} \mu_0.
\ee
Equation \eqref{eq:admissibility} makes sure that $P$ is admissible for \eqref{MSP} and we have
\beas
\int_0^1 \int_{\Omega} |\partial_{tt} X_t|^2  \De P\De t &\stackrel{\eqref{eq:phasespaceMongesol}}{=}&   \int_0^1 \int_{\RD \times \RD}|\partial_{tt} \mathbb{X}_t(x,v)|^2   \mu_{0}(x,v) \De x \De v  \De t\\
& \nonumber \stackrel{\eqref{eq:phasespaceflowmaps}}{=}&  \int_0^1 \int_{\RD \times \RD}|\partial_{t} \mathbb{V}_t(x,v)|^2   \mu_{0}(x,v) \De x \De v  \De t \\
&  \nonumber \stackrel{\eqref{eq:phasespaceflowmaps}}{=}&  \int_0^1 \int_{\RD \times \RD}|a_t (\mathbb{X}_t,\mathbb{V}_t)(x,v)|^2   \mu_{0}(x,v) \De x \De v  \De t\\
&  \nonumber \stackrel{\eqref{eq:phasespacepushforward}}{=}& \int_{0}^1 \int_{\RD \times \RD} |a_t (x,v)|^2  \mu_t(x,v) \De x \De v \De t. 
\eeas
Thus, given an optimal solution $(\mu,a)$ for \eqref{FDF}, we have constructed a feasible solution $P$ for \eqref{MSP} such that the cost function \eqref{FDFa} evaluated at $(\mu,a)$
equals the cost function \eqref{MSPa} evaluated at $P$. For the converse, we make the observation that  Theorem \ref{thm:MongeSol} grants the existence of an optimal Monge solution for the problem \eqref{eq:multimargphasespace}. We can lift this solution to an optimal Monge solution for \eqref{pathphasesp} using point \ref{item(ii)phasestr} of Theorem \eqref{thm:phasesplinestruc}. Therefore, there exist a measure  $\mu_0 \in \cP(\RD\times \RD)$ and two family of maps $(\bbX_t)_{t\in [0,1]}, (\bbV_t)_{t\in [0,1]}$ defined on $\RD \times \RD$ and taking values in $\RD$ such that
\begin{enumerate}[label= (\roman*),ref= (\roman*)]
\item $\bbX_0=\mathbf{id}_{\RD}, \bbV_0 = \mathbf{id}_{\RD}$ 
\item\label{item(ii)FDF}  $\partial_t \bbX_t = \bbV_t$ for all $t\in[0,1]$.
\item If we define $Q$ via
\[ Q= ((\bbX_t , \bbV_t)_{t \in [0,1]} )_{\#} \mu_0 , \]
then $Q$ is optimal for \eqref{pathphasesp}. In particular, this implies that if we define the plan $P$ via \eqref{eq:phasespaceMongesol}, then $P$ is optimal for \eqref{MSP}.
\end{enumerate}
Define now  $(\mu_t)$ as the marginal flow of $Q$, i.e.
\be\label{eq:mutdef} \forall t \in [0,1], \quad  \mu_t := (
\bbX_t , \bbV_t)_{\#} \mu_0. \ee
It is clear from the definition that $(\mu_t)$ satisfies \eqref{FDFc}. Moreover, provided the maps $\bbV_t$ are invertible, by setting 
\[ a_t(x,v) := (\partial_t \bbV_t) \circ (\bbV_t)^{-1} (x,v) \]
we obtain that \eqref{eq:phasespaceflowmaps} is satisfied. This implies that $w_t$, defined as in \eqref{eq:wt} satisfies the continuity equation for $(\mu_t)$, and hence that \eqref{FDFb} holds.
Hence $(\mu_t)$ is admissible for \eqref{FDF}. With a similar argument as above, we get
\beas
\int_{0}^1 \int_{\RD \times \RD} |a_t (x,v)|^2  \mu_t(x,v) \De x \De v \De t. 
 &\stackrel{\eqref{eq:mutdef}}{=}&   \int_0^1 \int_{\RD \times \RD}|a_t (\mathbb{X}_t,\mathbb{V}_t)(x,v)|^2   \mu_{0}(x,v) \De x \De v  \De t \\
& \nonumber \stackrel{\eqref{eq:phasespaceflowmaps}}{=}&  \int_0^1 \int_{\RD \times \RD}|\partial_{t} \mathbb{V}_t(x,v)|^2   \mu_{0}(x,v) \De x \De v  \De t \\
&  \nonumber \stackrel{\eqref{eq:phasespaceflowmaps}}{=}& \int_0^1 \int_{\RD \times \RD}|\partial_{tt} \mathbb{X}_t(x,v)|^2   \mu_{0}(x,v) \De x \De v  \De t\\
&  \nonumber \stackrel{\eqref{eq:phasespacepushforward}}{=}& \int_0^1 \int_{\Omega} |\partial_{tt} X_t|^2  \De P\De t
\eeas
Thus, starting for a particular optimal solution $P$ for \eqref{MSP} (precisely the one associated to the Monge solution of \eqref{pathphasesp}), we have constructed an admissible solution $(\mu,a)$ for \eqref{FDF} such that the cost function \eqref{MSPa} evaluated at $P$
equals the cost function \eqref{FDFa} evaluated at $(\mu,a)$.
\end{proof}

\section{A Riemannian geometry approach}\label{sec:riemannian}

There exist different approaches to the problem of interpolating smoothly data on a Riemannian manifold; in the upcoming discussion we shall follow the \emph{intrinsic approach}, see \cite{noakes1989cubic},\cite{camarinha2001geometry} and \cite{tahraoui2016riemannian} for infinite-dimensional manifolds.
Consider data $(t_i,x_i)_{i=0,\ldots,N}\subseteq [0,1] \times M$, where $M$ is a Riemannian manifold whose Levi-Civita connection is $\nabla$. Then, Holladay's theorem suggests to define the interpolating spline as the optimizer for
\bea\label{eq:Riemanniansplinegio}
\nonumber &{}&\inf_{X} \int_0^1 \big\langle \nabla_{\dot{X}_t} \dot{X}_t, \nabla_{\dot{X}_t} \dot{X}_t \big\rangle\,  dt\\
&{}&X \in  H^2 ([0,1];M), \ X_{t_i} = x_i  \quad  i =0, \ldots,N.
\eea

In a seminal paper \cite{otto2001geometry}, Otto discovered that the metric space $(\cP_2(\RD), W_2(\cdot,\cdot))$ can be looked at almost as an infinite dimensional Riemannian manifold. In the next subsection we shall present a formal construction of the Riemannian metric for $(\cP_2(\RD)$ (often called the Otto metric).  But our claims will not be rigorously detailed and our treatment of the subject will only be  partial; to gain a deeper insight we refer the reader to Otto's paper and, in addition, to \cite{giglisecond},\cite{ambrosio2013user},\cite{lott2008some},\cite{villani2008optimal}.

\subsection{The Riemannian metric of optimal transport}
Aim of this subsection is to define formally a kind of Riemannian metric on $\cP_2(\R^d)$ for which displacement interpolations \cite{mccann1997convexity} are constant speed geodesics. The construction begins by identifying the tangent space at $\rho$ with the space of square integrable gradient vector fields. The identification is possible thanks to Brenier's theorem \cite{brenier1991polar}. This space is
\[ \tsp_{\rho} := \overline{\left\{ \nabla \varphi; \varphi \in C^{\infty}_c  \right\}}^{L^2(\rho)}. \]
The second step is to define the first derivative (velocity field) $v_t \in \tsp_{\rho_t}$  of a curve $(\rho_t)$ through the continuity equation 
\[\partial_t \rho_t + \nabla \cdot ( v_t \rho_t)=0, \quad v_t \in \tsp_{\rho_t}. \]
 Then, one defines the Riemannian metric by means of the $L^2$ product 
\be\label{eq:riemanmetr} \langle \nabla \varphi, \nabla \psi \rangle_{\mathbf{T}_{\rho}} := \int_{\R^d} \langle \nabla \varphi(x), \nabla \psi (x)\rangle \, \rho(x) \De x, \ee
where $\langle ., . \rangle$ stands for the standard inner product on $\R^d$. The Benamou-Brenier formula \eqref{BB} establishes that the displacement interpolation is a constant speed geodesic for this Riemannian structure, as it minimizes the energy functional among all curves with a given start and end. If we denote $\covdev$ the Levi Civita connection associated with the Riemannian metric, it turns out that, if $(\rho_t)$ is a smooth curve and $(v_t)$ its velocity field, then the covariant derivative of $(v_t)$ along $(\rho_t)$ is given by the formula (see e.g. \cite[Example 6.7]{ambrosio2013user})
\begin{equation*}
\covdev_{v_t} v_t = \partial_t v_t + \frac{1}{2}\nabla |v_t|^2  \in \tsp_{\rho_t}.
 \end{equation*} 
 Thus, we have
 \be\label{eq:covdevnorm} \langle\covdev_{v_t} v_t, \covdev_{v_t} v_t \rangle_{\tsp_{\rho_t}} = \int_{\RD} \Big|  \partial_t v_t + \frac{1}{2}\nabla |v_t|^2    \Big|^2 \rho_t ( x ) \De x . \ee
 \subsection{An alternative definition for measure-valued splines}
In view of \eqref{eq:Riemanniansplinegio} and \eqref{eq:covdevnorm} it would be natural to define measure-valued splines by looking at
\begin{subequations}\label{eq:RiemannianOTspline}
\begin{align}
&{}& \inf_{\rho,v} \int_0^1 \int_{\RD} \Big|\partial_t v_t + \frac{1}{2}\nabla |v_t|^2\Big|^2(x)  \rho_t(x) \De x dt\\
&{}&\ \partial_t \rho_t(x) + \nabla \cdot (v_t \rho_t)(x)=0 , \quad v_t \in \tsp_{\rho_t} \\
&{}&(\rho_t) \in H^2([0,1] ; \cP_2(\RD)), \ \rho_{t_i}= \rho_i, \quad i=0, \ldots, N,
\end{align}
\end{subequations}
where the space $H^2([0,1] ; \cP_2(\RD))$ should be properly defined using the notions of absolutely continuous and regular curve (\cite[Ch. 6]{ambrosio2013user}). Clearly the problem \eqref{eq:RiemannianOTspline} looks rather different from \eqref{FDF}, and therefore, it should not be equivalent to \eqref{MSP}. However, it seems that, although different, the two problems are strongly related: in the next subsection, we shall provide a heuristic showing that \eqref{MSP} can be viewed as a relaxation of \eqref{eq:RiemannianOTspline}.
\subsection{The problem \eqref{eq:RiemannianOTspline} and the Monge formulation of \eqref{MSP}} 
We have seen that Monge solutions exist for the relaxation \eqref{eq:multimargphasespace}. Using point \ref{item(ii)phasestr} of Theorem \ref{thm:phasesplinestruc}, those Monge solutions can be lifted to path space to obtain Monge solutions for \eqref{pathphasesp}.  However, a Monge solution for \eqref{MSP} has a different structure. A Monge solution for \eqref{MSP} is a plan $P$ for which there exist a family of maps $\bbX_t:\RD \rightarrow \RD$ such that 
\be\label{eq:Mongeplanform} P:=  ((\bbX_t)_{t \in [0,1]})_{\#} \rho_0 .\ee
It is important to note the difference between \eqref{eq:Mongeplanform} and \eqref{eq:phasespaceMongesol}. Here, the maps $\bbX_t$ are defined on $\RD$ and $\rho_0\in \cP(\RD)$; there the maps $\bbX_t$ are defined on $\RD \times \RD$ and $\mu_0 \in\cP(\RD \times \RD)$. \\
Let us now present a heuristic connecting \eqref{eq:RiemannianOTspline} with the Monge formulation of \eqref{MSP}. Consider a solution $(\rho,v)$ for \eqref{eq:RiemannianOTspline}, and defines the maps $\mathbb{X}_t$ via 
\be\label{eq:flowmaps} \partial_t \mathbb{X}_t (x)= v_t(\mathbb{X}_t(x)), \quad \mathbb{X}_0(x)=x. \ee
and $P$ through \eqref{eq:Mongeplanform}.
These are the flow maps for the velocity field $(v_t)$ on $\RD$ and satisfy
\be\label{eq:pushforward}    (\mathbb{X}_t)_{\#}\rho_0 = \rho_t  \ee
 Therefore, $P$ is admissible for \eqref{MSP} and we have
\beas
\int_0^1\int_{H^2} | \partial_{tt}X_t |^2 \De P  \De t &\stackrel{\eqref{eq:Mongeplanform}}{=}& \int_0^1 \int_{\RD} |\partial_{tt} \mathbb{X}_t(x) |^2 \rho_0( x) \De x \De t \\
&\stackrel{\eqref{eq:flowmaps}}{=}&
  \int_0^1 \int_{\RD} \Big|\partial_t v_t (\bbX_t(x))\Big|^2  \rho_0(x) \De x \De t\\
&\stackrel{\eqref{eq:flowmaps}}{=}&  \int_0^1 \int_{\RD} \Big|\partial_t v_t + D_{v_t} v_t \Big|^2(\mathbb{X}_t(x))  \rho_0(x) \De x \De t\\
&=& \int_0^1 \int_{\RD} \Big|\partial_t v_t + \frac{1}{2}\nabla |v_t|^2\Big|^2(\bbX_t(x))  \rho_0 \De x \De t\\
&\stackrel{\eqref{eq:pushforward}}{=}& \int_0^1 \int_{\RD} \Big|\partial_t v_t + \frac{1}{2}\nabla |v_t|^2\Big|^2(x)  \rho_t(x) \De x \De t,
\eeas
where we denoted by $D_{v_t} v_t$ the Jacobian of the vector field $v_t$ applied to $v_t$. Since $v_t$ is of gradient type, we have indeed $D_{v_t} v_t = \frac{1}{2}\nabla |v_t|^2 $.
Thus, we have seen that, to a solution of \eqref{eq:RiemannianOTspline} we can associate a Monge solution for \eqref{MSP} and the cost of the two solutions for their respective problems is identical. To conclude that the two problems are equivalent, we should reverse this last statement. But to do this, we should know that we can w.l.o.g consider Monge solutions \eqref{eq:Mongeplanform} where the maps $\bbX_t$ are the flow maps for a \emph{gradient} vector field. We do not know, at the moment, whether this is true or not. If we remove the constraint $v_t \in \tsp_{\rho_t}$ from \eqref{eq:RiemannianOTspline}, then it is natural to conjecture that \eqref{eq:RiemannianOTspline} and the Monge formulation of \eqref{MSP} are equivalent.

 \section{Proofs}\label{sec:proofs}
\begin{proof}[Proof of theorem \ref{thm:splinestruc}]
We first prove that \ref{item(ii)}$\Rightarrow$\ref{item(i)}. Let $\hat{P}$ be as in \ref{item(ii)},  $P$ any other admissible probability measure for \eqref{MSP} and $\pi := (X_{\mathcal{T}})_{\#} P$. Observe that, since $P$ is supported on $H^2$ we have that 
$\omega$ is almost surely an admissible path for the problem \eqref{eq:spline} for the choices $x_i =X_{t_i}(\omega)$. Therefore 
\be\label{eq:splinestruc2} P -\text{a.s.} \quad \int_{0}^1 |\partial_{tt} X_{t}|^2 \De t \geq \int_{0}^1 |\partial_{tt} S_t(X_{\mathcal{T}}) |^2 \De t,  \ee
where we recall that $X_{\mathcal{T}} = (X_{t_0}, \ldots, X_{t_N})$ and $S(x_0,\ldots,x_N)$ is the natural interpolating spline.
Using this, we get
\bea\label{eq:splinestruc3} 
  \int_{0}^1 \int_{\Omega}|\partial_{tt}X_{t}|^2 \De P\De t  &\geq &  \int_{0}^1 \int_{\Omega}|\partial_{tt} S_t(X_{\mathcal{T}}) |^2 \De P\, \De t\\
&=& \nonumber \int_{\Omega} \mathcal{C}\left(X_{\mathcal{T}}\right) \De P  \\
&=& \nonumber \int \mathcal{C}(x_0,\ldots,x_N) \De \pi\\
&\geq & \int\mathcal{C}(x_0,\ldots,x_N) \De \hat{\pi}
 \eea
 where the last inequality comes from the optimality of $\hat{\pi}$. On the other hand, since $\hat{P}(\spsp^0)=1$, we have that 
 \[\hat{P}-\text{a.s.}, \quad X_{\cdot}= S_{\cdot}(X_{\mathcal{T}}) \]
Thus, 
\bea\label{eq:splinestruc1}
 \int_{0}^1 \int_{\Omega}|\partial_{tt}X_{t}|^2 \De \hat{P}\De t  &= &  \int_{0}^1 \int_{\Omega}|\partial_{tt} S_t(X_{\mathcal{T}}) |^2 \De \hat{P}\, \De t\\
&=& \nonumber \int_{\Omega} \mathcal{C}\left(X_{\mathcal{T}}\right) \De \hP  \\
&=& \nonumber \int \mathcal{C}(x_0,\ldots,x_N) \De \hat{\pi},
\eea
which proves that $\hat{P}$ is an optimal solution for \eqref{MSP}.

Let us now prove \ref{item(i)}$\Rightarrow$ \ref{item(ii)}. Let $\hat{P}$ an optimal measure for \eqref{MSP}, and assume that $\hat{P}(\spsp^0) <1$.  Consider the Markov kernel $\mathcal{K}$ 
 \bes \mathcal{K}: \R^{d \times (N+1)} \times  \mathcal{B}(\Omega) \rightarrow [0,1], \quad  \mathcal{K}(x_0,\ldots,x_N,A) = \begin{cases} 1, \quad & \mbox{if $ S_{\cdot}(x_0,\ldots,x_N) \in A $ } \\ 0, \quad & \mbox{ otherwise } \end{cases}  \ees
 and define $P^*$ by composing $\hat{\pi}$ with  $\mathcal{K}$
 \be\label{eq:splinestruc5}  P^*(A) = \int \mathcal{K}(x_0,\ldots,x_N,A)  \De \hat{\pi}.  \ee
 By construction, we have that $P^*(\spsp^0)=1$ and that $(X_{\mathcal{T}})_{\#} P^*=\hat{\pi}$. Thus, arguing as in \eqref{eq:splinestruc1} we obtain that 
 \be\label{eq:splinestruc4} \int_{0}^1 \int_{\Omega}|\partial_{tt} X_{t}|^2 \De P^* \, \De t   = \int \mathcal{C}(x_0,\ldots,x_N) \De \hat{\pi}. \ee
 Moreover, we have that \eqref{eq:splinestruc2} holds under $\hat{P}$ since $\hat{P}$ is admissible for \eqref{MSP} and since  $\hP(\spsp^0)<1$ we also have the strict inequality in \eqref{eq:splinestruc2} holds with positive probability under $\hP$. Arguing as in \eqref{eq:splinestruc3}, we obtain, with  minimal changes,
\[ \int_{0}^1\int_{\Omega} |\partial_{tt} X_t |^2 \De \hP \De t > \int\mathcal{C}(x_0,\ldots,x_N) \De \hat{\pi} . \]
This last inequality, toghether with \eqref{eq:splinestruc4} contradicts the optimality of $\hat{P}$. Thus, it must be that $\hP\left(\spsp^0\right)=1$, which also implies that $P=P^*$. Arguing again as above, it is easy to see that 
\[  \int_{0}^1 \int_{\Omega}|\partial_{tt} X_{t}|^2 \De \hat{P}\De t   = \nonumber \int \mathcal{C}\left(x_0, \ldots, x_N \right) \De \hat{\pi}\]
Assume now that $\hat{\pi}$ is not an optimal measure for \eqref{eq:mulitmargprob}. Then there exists $\pi^*$ which is admissible and performs better than $\hat{\pi}$. We can again define $P^*$ as in \eqref{eq:splinestruc5} replacing $\hat{\pi}$  with $\pi^*$. Reasoning as in the previous cases we get
\beas      \int_{0}^1  \int_{\Omega}|\partial_{tt}X_t |^2 \De \hP \De t &\geq & \int \mathcal{C}(x_0,\ldots, x_N) \De \hat{\pi}\\
& >&  \int\mathcal{C}(x_0,\ldots,x_N) \De \pi^*=    \int_{0}^1  \int_{\Omega}|\partial_{tt} \omega_t |^2 \De P^* \De t       ,
\eeas
which contradicts the optimality of $\hP$. Thus, it must be that $\hat{\pi}$ is optimal for \eqref{eq:mulitmargprob}. The proof that \ref{item(i)}$\Rightarrow$\ref{item(ii)} is now concluded.

Finally, let us show that an optimal solution to \eqref{MSP} exists. Notice that the function $\cC$ is a quadratic form and therefore an optimal solution $\hat{\pi}$ to \eqref{eq:mulitmargprob} always exists. The proof of this is straightforward adaptation of the proof of \cite[Th 1.5]{ambrosio2013user}. If we construct $\hat{P}$ as in \eqref{eq:splinestruc5}, then the implicaton  \ref{item(ii)}$\Rightarrow$\ref{item(i)} yields the conclusion.
\end{proof}
\begin{proof}[Proof of Lemma \ref{lm:costref}]
Consider the problem obtained by looking only at the time interval $[t_i,t_{i+1}]$, i.e.
\begin{subequations}\label{eq:spline+vel+restr}
\begin{align}
&{}&\inf_{X,V} \int_{t_i}^{t_{i+1}} |\dot{V}_t|^2 dt\\
&{}&(X,V) \in H^1([t_i,t_{i+1}];\RD) \times  H^1([t_i,t_{i+1}];\RD )  \\
&& \dot{X}_t=V_t, \quad \forall t \in [t_i,t_{i+1}]\\
&& X_{t_j} = x_j \quad j=i,i+1, \\
&& V_{t_j} = v_j, \quad j=i,i+1.
\end{align}
\end{subequations}
Using a standard argument based on integration by parts it is seen that the optimal solution is the only admissible $(\hat{X}^i,\hat{V}^i)$ such that $\hat{X}^{i} \in \Pi_{3}([t_i,t_{i+1}])$. A standard calculation then also proves that the optimal value for \eqref{eq:spline+vel+restr} is $(t_{i+1}-t_{i})^{-1} c(x_{i},x_{i+1},v_i,v_{i+1})$. More details can be found in Section \ref{sec:gaussian}. Next, we define  
\[ \forall t \in [0,1] \quad \hat{X}_t: = \sum_{i=0}^{N-1} \hat{X}^i_t  \mathbf{1}_{[t_{i}, t_{i+1})}(t) +x_{t_N}\mathbf{1}_{t=1} \]
\end{proof}
By construction, $\hat{X}$ of class $C^1$ on $[0,1]$ and on each interval $[t_{i},t_{i+1}]$ the second derivative exists and is bounded. This implies that $\hat{X}$ is in $H^2([0,1];\RD)$, and that $(\hat{X},\hat{V})$ is admissible for \eqref{eq:spline+vel}, where we set $\hat{V}:= \dot{\hat{X}}$ . The optimality follows observing that for any other admissible solution $(X,V)$ we have, using the optimality of $(\hat{X}^i,\hat{V}^i)$
\[  \int_{0}^1  |\dot{V}_t|^2 \De t =  \sum_{i=0}^{N-1} \int_{t_i}^{t_{i+1}}  |\dot{V}_t|^2 \De t \geq \sum_{i=0}^{N-1} \int_{t_i}^{t_{i+1}}  |\dot{\hat{V}}^i_t|^2 \De t =  \int_{0}^1  |\dot{\hat{V}}_t|^2 \De t . \]
This shows that $(\hat{X},\hat{V})$ is optimal, from which \eqref{eq:explicitcost} follows. \eqref{eq:costref1} follows from \eqref{eq:explicitcost} and Holladay's Theorem.
\begin{proof}[Proof of Theorem \ref{thm:phasesplinestruc}]
In view of Lemma \ref{lm:costref}, the proof of this theorem is a straightforward adaptation of that of Theorem \ref{thm:splinestruc}.
\end{proof}
\begin{proof}[Proof of Propostion \ref{lem:pahsespaceform}]
Let us make the preliminary observations that, because of \eqref{eq:costref1}, if $\pi$ is admissible for \eqref{eq:mulitmargprob}, and we define 
\be\label{eq:pitogamma} \gamma:= (X_0,\ldots,X_N,\mathbb{V}(X_0,\ldots,X_N))_{\#} \pi,\ee then we have
\be\label{eq:pitogamma2} \int \mathcal{C}(x_0,\ldots, x_N) \De \pi = \sum_{i=0}^{N-1} (t_{i+1}-t_{i})^{-1} \int c(x_{i},x_{i+1},v_{i},v_{i+1} ) \De \gamma. \ee
On the other hand, if $\gamma$ is any admissible plan for \eqref{eq:multimargphasespace} and we define 
\be\label{eq:gammatopi}\pi:=(X_0,\ldots, X_N)_{\#}\gamma, \ee 
then 
\be\label{eq:gammatopi2} \int \mathcal{C}(x_0,\ldots,x_N) \De \pi \leq \sum_{i=0}^{N-1} (t_{i+1}-t_{i})^{-1} \int c(x_{i},x_{i+1},v_{i},v_{i+1} ) \De \gamma.  \ee
We begin by proving \ref{item(i)phase}. Let $\hat{\gamma}$ be optimal for \eqref{eq:multimargphasespace} and consider an admissible coupling $\pi$ for \eqref{eq:mulitmargprob}. Define $\gamma$ through \eqref{eq:pitogamma}. Then we have, by optimality of $\hat{\gamma}$
\beas \int \mathcal{C}(x_0,\ldots,x_N) \De \pi &\stackrel{\eqref{eq:pitogamma2}}{=}& \sum_{i=0}^{N-1} (t_{i+1}-t_{i})^{-1} \int c(x_{i},x_{i+1},v_{i},v_{i+1} ) \De \gamma \\
& \geq &   \sum_{i=0}^{N-1} (t_{i+1}-t_{i})^{-1} \int c(x_{i},x_{i+1},v_{i},v_{i+1} ) \De \hat{ \gamma } 
\\ &\stackrel{\eqref{eq:gammatopi2}}{\geq}& \int \mathcal{C}(x_0,\ldots,x_N) \De \hat{\pi},
 \eeas
 and therefore $\hat{\pi}$ is optimal for \eqref{eq:multimargcost}. To prove \ref{item(ii)}, assume that $\hat{\pi}$ is optimal  for \eqref{eq:mulitmargprob} and let $\gamma$ be admissible for \eqref{eq:multimargphasespace}. Define $\pi$ via \eqref{eq:gammatopi}. Then we have, by optimality of $\hat{\pi}$
 \beas  \sum_{i=0}^{N-1} (t_{i+1}-t_{i})^{-1} \int c(x_{i},x_{i+1},v_{i},v_{i+1} ) \De \gamma &\stackrel{\eqref{eq:gammatopi2}}{\geq}&  \int \mathcal{C}(x_0,\ldots,x_N) \De \pi   \\
& \geq &  \int \mathcal{C}(x_0,\ldots,x_N) \De \hat{\pi} \\ &\stackrel{\eqref{eq:pitogamma2}}{=}& \sum_{i=0}^{N-1} (t_{i+1}-t_{i})^{-1} \int c(x_{i},x_{i+1},v_{i},v_{i+1} ) \De \hat{ \gamma },
 \eeas
 which yields the conclusion.
\end{proof}

\begin{proof}[Proof of Theorem \ref{thm:MongeSol}]
Let $\hat{\gamma}$ be optimal for \eqref{eq:multimargphasespace}. For each $i=0,..,N-1$ consider the reduced problem
\begin{eqnarray}\label{eq:MongeSol2}
\inf_{\pi} \int c(x_{i},x_{i+1},v_{i},v_{i+1}) \De \pi \\
\nonumber \pi \in \cP(\R^{d \times 2} \times \R^{d \times 2}), \quad (X_j,V_j)_{\#} \pi = \hgamma_i, \quad j =i,i+1,
\end{eqnarray}
where the projection maps $X_j,V_j, j=i,i+1$ are defined in the obvious way on $\R^{d \times 2} \times \R^{d\times 2}$. It is rather easy to see that the cost $c$ satisfies the ``twist condition" $(A1)$ from \cite{pass2015multi}. See also \cite{CheGeoPav17} for an alternative proof. Thus, since all the $\hgamma_i$ are absolutely continuous, we can use \cite[Th. 2.21]{pass2015multi} to conclude that there exists a unique solution $\tilde{\pi}_i$ to \eqref{eq:MongeSol2}, and that the solution is in Monge form. Thus there exists a map $F_{i}: \R^{d \times 2} \rightarrow \R^{d \times 2}$ such that
\begin{equation}\label{eq:MongeSol3}
\tilde{\pi}_i \Big( (X_{i+1},V_{i+1})= F_{i}( X_{i},V_{i}) \Big) =1. 
\end{equation}
For all $i=0,\ldots,N$, define the maps $(\varphi_i,\psi_i)$ via 
\begin{equation}\label{eq:MongeSol4}(\varphi_0,\psi_0)=(\mathbf{id}_{\RD},\mathbf{id}_{\RD}), \quad  (\varphi_i,\psi_i) := F_{i-1} \circ (\varphi_{i-1},\psi_{i-1}),
\end{equation}
Next, define $\tilde{\gamma}$ by
\[ \tilde{\gamma} = ( \mathbf{id}_{\RD}, \varphi_1,\ldots,\varphi_N,\mathbf{id}_{\RD},\psi_1,\ldots, \psi_N )_{\#} \hat{\gamma}_0 .\]
By construction $\tilde{\gamma}$ is admissible for \eqref{eq:multimargphasespace} and that $(X_i,V_i,X_{i+1},V_{i+1})_{\#} \tilde{\gamma} = \tilde{\pi}_i$ for all $i=0,\ldots,N-1$.
Since for any $i$,  $\hat{\pi}_i:=(X_{i},V_{i},X_{i+1},V_{i+1})_{\#} \hgamma$ is admissible  for \eqref{eq:MongeSol2} we have 
\[ \int c(x_{i},x_{i+1},v_{i},v_{i+1}) \De \hat{\pi}_i \geq  \int c(x_{i},x_{i+1},v_{i},v_{i+1}) \De \tilde{\pi}_i.\] 
Assume now that $\hat{\pi}_j \neq \tilde{\pi}_j$ for some $j$. Then, since  \eqref{eq:MongeSol2}  admits a unique optimal solution we have:
\[\int c(x_{j},x_{j+1},v_{j},v_{j+1}) \De  \hat{\pi}_j  >  \int c(x_{j},x_{j+1},v_{j},v_{j+1}) \De \tilde{\pi}_j \] 
But this would imply that
\begin{eqnarray*}
\sum_{i=0}^{N-1} (t_{i+1}-t_{i})^{-1} \int c(x_{i},x_{i+1},v_{i},v_{i+1} ) \De \hgamma &=& \sum_{i=0}^{N-1} (t_{i+1}-t_{i})^{-1} \int c(x_{i},x_{i+1},v_{i},v_{i+1} ) \De \hat{\pi}_i \\
&>& \sum_{i=0}^{N-1} (t_{i+1}-t_{i})^{-1} \int c(x_{i},x_{i+1},v_{i},v_{i+1} ) \De \tilde{\pi}_i \\
&=& \sum_{i=0}^{N-1} (t_{i+1}-t_{i})^{-1} \int c(x_{i},x_{i+1},v_{i},v_{i+1} ) \De \tilde{\gamma},
\end{eqnarray*}
which contradicts the optimality of $\hgamma$. Therefore, $\hat{\pi}_i=\tilde{\pi}_i$ for all $i$, which yields the conclusion, using \eqref{eq:MongeSol3} and \eqref{eq:MongeSol4} recursively.
\end{proof}
\begin{proof}[Proof of Propositon \ref{prop:threepoints}]
We assume w.l.o.g., $t_1-t_0=t_2-t_1$. It can be computed explicitly that, up to a  positive multiplying constant, $\cC(x_0,x_1,x_2) = |x_2- 2 x_1 +x_0|^2$. Fix now a point $\textbf{x}=(x_0,x_1,x_2)$ in the support of an optimal solution $\hat{\pi}$. Combining Th. 2.3 and Eq. (3) from \cite{pass2012structure} we obtain that if we denote $q_+$ the number of positive eigenvalues of the block-matrix
\[  \begin{pmatrix}  0 & D_{x_0,x_1} \cC(\textbf{x})  & D_{x_0,x_2} \cC(\textbf{x})   \\ D_{x_1,x_0} \cC(\textbf{x})  & 0 &  D_{x_1,x_2} \cC(\textbf{x})  \\  D_{x_2,x_0} \cC(\textbf{x})&  D_{x_2,x_1} \cC(\textbf{x})& 0 \end{pmatrix}   \]
where $ D_{x_i,x_j} \cC(\textbf{x}) $ is the $d \times d$ matrix given by $ (D_{x_i,x_j} \cC(\textbf{x}))_{kl}= (\partial_{x^k_i,x^l_j}  \cC)(\textbf{x})$, then the support of $\pi$ is locally  of dimension $3 d -q_+$ around $\textbf{x}$. Given the form of $\cC$ we have that
\[ (D_{x_0,x_1} \cC(\textbf{x}))_{kl} =  (D_{x_1,x_2} \cC(\textbf{x}))_{kl} = - 4 \delta_{kl}, \quad D_{x_0,x_2}\cC(\textbf{x}) = 2 \delta_{kl}\]
where $\delta_{kl}$ is the Kronecker delta. The conclusion then follows from a direct calculation.
\end{proof}

\section{The Gaussian case}\label{sec:gaussian}
%

We specialize and discuss the case where all the marginal distributions are Gaussian distributions on $\RD$, with the $i$th marginal $\rho_i$ having mean $m_i$ and covariance $\Sigma_i$ for $0\le i\le N$, denoted by $\rho_i= N(m_i,\Sigma_i)$. For simplicity, we take $t_i = i, ~0\le i \le N$. 
It turns out that the interplating one-time marginals are also Gaussian and, in fact, problem \eqref{MSP} easily decouples into interpolating separately means and covariances. Dealing with the means requires constructing a cubic spline that interpolates only the means $m_0,\ldots,m_N$ at the sample points. This cubic spline is denoted by $m(t), 0\le t\le N$. Interpolating the covariances requires solving a semidefinite program (SDP) as we explain next. 

We cast the problem in phase space as already done in Section \ref{sec:phase_space}.
 A cubic spline $X(\cdot)$, which solves \eqref{eq:spline}, also solves 
 	\begin{subequations}\label{eq:liftcurve}
	\begin{eqnarray}
	\inf_{X,V} && \int_{0}^{N} |\dot V_t|^2 dt \\
	&& \dot X_t = V_t, ~~X_{t_i} = x_i, ~0\le i \le N.
	\end{eqnarray}
	\end{subequations}
The optimality conditions can be written in the form
	\begin{subequations}
	\begin{eqnarray}
	\dot{X}_t &=& \phantom{-}V_t
	\\
	\dot{V}_t &=& \phantom{-}\Lambda_t\\
	\dot{\Lambda}_t&=&-M_t
	\end{eqnarray}
	\end{subequations}
where $\Lambda,M$ are Lagrange multipliers, and $M$ is piecewise constant in the specified intervals; clearly, $X\in C^2$. 

Earlier, we indicated that the cost \eqref{eq:explicitcost}, which involves all $(x_i,v_i)$'s, can be optimized over the $v_i$'s to derive 
$\mathcal{C}(x_0,\ldots,x_N)$ in \eqref{eq:costref1}, which is quadratic, say\footnote{We denote by $\,^T$ the ``transpose of''.},
\[
		\cC(x_0,x_1,\cdots,x_N) = x^T R x
\]
for a positive semidefinite $R$ and $x=(x_0^T,\ldots,x_N^T)^T$, considering the $x_i$'s as column vectors.  

Hence, our problem becomes
	\begin{equation}
	\inf \{\mathbb{E} \{X'RX\}~\mid~X=(X_{t_0}^T,\ldots,X_{t_N}^T)^T \mbox{ with }X_{t_i} \sim {\mathcal N}(m_i,\Sigma_i),~0\le i\le N\}
	\end{equation}
	over a choice of  correlation between the $X_{t_i}$'s so that each is normal with the specified mean and covariance, and the cost is minimized.
The minimum corresponds to iterpolating the means via a spline, as indicated earlier, and solving the SDP
	\begin{equation}\label{eq:directSDP}
	\inf_{\Sigma\ge 0} \{{\rm Tr}(R\Sigma)~\mid~\Sigma(i,i)= \Sigma_i,~0\le i\le N\}
	\end{equation}
to obtain required correlations between different points in time.
Thus, $X$ can be taken to be Gaussian. Here, $\Sigma(i,i)$ denotes successive $d\times d$-diagonal-block entries of the correlation matrix $\Sigma$ of $X$.

An alternative formulation, which is easier to encode and compute, is to consider minimizing directly
\eqref{eq:explicitcost} over a choice of joint covariance of all $(X_{t_i},V_{t_i})$'s, subject of course to the $X_{t_i}$'s being normal with the specified covariances.
Indeed, $c(x_i,x_{i+1},v_i,v_{i+1})$ in \eqref{eq:explicitcost_formula} takes the form
	\[
		(\xi_{i+1}-\Phi \xi_i)^T Q(\xi_{i+1}-\Phi \xi_i)
	\]
where $\xi_i=(x_i^T,v_i^T)^T$ and
	\[
		\Phi =
		\left[\begin{matrix}
		1 & 1\\ 0 & 1
		\end{matrix}\right]\otimes I_d,
		\quad 
		Q =
		\left[\begin{matrix}
		12 & -6 \\ -6 & 4
		\end{matrix}\right]\otimes I_d.
	\]
Now, denoting $\Xi_i=(X_i^T,V_i^T)^T$, 
	\[
		\mathbb{E}\{\Xi_i\Xi_i^T\} = \hat\Sigma_i
		\quad
		\mbox{and}
		\quad
		\mathbb{E}\{\Xi_i\Xi_j^T\} = S_{i,j}
	\]
for all $i,j$, the cost becomes
	\[
	\mathbb{E}\{ \sum_{i=0}^{N-1} c(x_i,x_{i+1},v_i,v_{i+1})\}
	=
	\sum_{i=0}^{N-1} {\rm Tr}(Q\hat\Sigma_{i+1}+\Phi^TQ\Phi\hat\Sigma_{i}-2Q\Phi S_{i,i+1}).
	\]
The covariance of the vector of $\Xi$'s will be denoted by
\begin{equation}
\hat{\Sigma} = \left[ \begin{matrix} \hat{\Sigma}_0  & S_{0,1} & \ldots & S_{0,N}\\
   S^T_{0,1} & \hat{\Sigma}_1 & \ldots & S_{1,N}\\
   \vdots & \vdots & \ddots & \vdots \\
   S_{0,N}^T & S_{1,N}^T & \cdots & \hat\Sigma_{N}
\end{matrix}\right],
\end{equation} 
and the optimization in \eqref{eq:multimargphasespace} now becomes
	\begin{subequations}\label{eq:noteffSDP}
	\begin{eqnarray}\label{eq:noteffSDP1}
		c_{\rm opt}= &&\inf\left\{\sum_{i=0}^{N-1} {\rm Tr}(Q\hat\Sigma_{i+1}+\Phi^TQ\Phi\hat\Sigma_{i}-2Q\Phi S_{i,i+1})\right.\mid\\
		&& \left.\hspace*{3cm} \hat\Sigma_{i}=\left[\begin{matrix}
		\Sigma_i & A_i \\ A_i^T & B_i
		\end{matrix}\right] \mbox{ and } \hat{\Sigma} \geq 0\right\}.\label{eq:noteffSDP2}		
	\end{eqnarray}
	\end{subequations}
Interestingly, the constraint can be simplified and the problem becomes
	\begin{subequations}\label{eq:effSDP}
	\begin{eqnarray}\label{eq:effSDP1}
		c_{\rm opt}= &&\inf\left\{\sum_{i=0}^{N-1} {\rm Tr}(Q\hat\Sigma_{i+1}+\Phi^TQ\Phi\hat\Sigma_{i}-2Q\Phi S_{i,i+1})\right.\\
		&& \left.\hspace*{3cm}\left[\begin{matrix}
		\hat \Sigma_i & S_{i,i+1} \\ S_{i,i+1}^T & \hat \Sigma_{i+1} 
		\end{matrix}\right] \ge 0,~~ \hat\Sigma_{i}=\left[\begin{matrix}
		\Sigma_i & A_i \\ A_i^T & B_i
		\end{matrix}\right]\right\}. \label{eq:effSDP2}
	\end{eqnarray}
	\end{subequations}
To see this, we first note that the cost is independent of $S_{i,j}$ for $|j-i|>1$. Moreover, \eqref{eq:noteffSDP2} implies \eqref{eq:effSDP2}. Therefore, to show the equivalence, we need only to prove that for any $\hat\Sigma_0,\ldots,\hat\Sigma_N,S_{0,1},\ldots,S_{N-1,N}$ satisfying \eqref{eq:effSDP2} there always exists $\hat\Sigma$ fulfilling \eqref{eq:noteffSDP2}. This can be done in a constructive manner. We construct a graphical model of $N+1$ random vectors $\Xi_0, \Xi_1,\ldots, \Xi_N$ such that $\Xi_{i+1}, \Xi_{i-1}$ are conditionally independent given $\Xi_i$ for each $i$, i.e., that the probability density of these vectors factors 
	\[
		p(\Xi_0, \Xi_1,\ldots, \Xi_N) = p(\Xi_0) p(\Xi_1\mid \Xi_0)\cdots p(\Xi_N\mid \Xi_{N-1}).
	\]
In addition, we let $p(\Xi_0)$ be a Gaussian density with zero mean and covariance $\hat \Sigma_0$, and $p(\Xi_{i+1}\mid\Xi_i)$ be a Gaussian density with mean $S_{i,i+1}^T\hat\Sigma_i^\dagger \Xi_i$ and covariance 
\[
\hat\Sigma_{i+1}-S_{i,i+1}^T\hat\Sigma_i^\dagger S_{i,i+1}.
\]
 Here $\dagger$ denotes pseudo-inverse. Under \eqref{eq:effSDP2}, the above constructing process is valid. Now we observe that $\Xi_i$ is a zero-mean Gaussian random vector with covariance $\hat\Sigma_i$ and $\mathbb{E}\{\Xi_i\Xi_{i+1}^T\} = S_{i,i+1}$. The proof follows by induction. Finally, let $\hat\Sigma$ denote the covariance matrix of the random vector $[\Xi_0^T, \Xi_1^T,\ldots, \Xi_N^T]^T$. It follows that it satisfies \eqref{eq:noteffSDP2}.

The formulation \eqref{eq:effSDP} is a SDP problem that can be solved efficiently for reasonably large size. The complexity scales linearly as the number $N$ of marginals increase. This is the essential difference twith \eqref{eq:directSDP}, where the complexity scales as $N^6$ in the worst case. 

For fixed $\hat\Sigma_i$, minimizing the cost over $S_{i,i+1}$ is equivalent to solving $N$ separate generalized optimal mass transport problems \cite{CheGeoPav17}. Thus, the optimal solution induces a one-to-one linear map from $\Xi_i$ to $\Xi_{i+}$, which implies that the $4d$ by $4d$ matrix 
	\begin{equation}\label{eq:block}
		\left[\begin{matrix}
		\hat \Sigma_i & S_{i,i+1} \\ S_{i,i+1}^T & \hat \Sigma_{i+1} 
		\end{matrix}\right] 
	\end{equation}
is of rank at most $2d$. Now we repeat the above constructing strategy when we proved the equivalence between \eqref{eq:noteffSDP2} and \eqref{eq:effSDP2}. Since \eqref{eq:block} is of rank at most $2d$, the relation between $\Xi_{i+1}$ and $\Xi_i$ is deterministic, and therefore the covariance corresponding to $p(\Xi_{i+1}\mid\Xi_i)$ is $0$, from which we deduce that the matrix $\hat\Sigma$ that we constructed is of rank at most $2d$. Hence, we have established the following statement.
\begin{lemma}
There exists at least one solution $\hat{\Sigma} \in \mathbb{R}^{2d(N+1) \times 2d(N+1)}$ of the optimization in \eqref{eq:noteffSDP} having rank at most $2d$.
\end{lemma}

Finally, the optimal selection for the covariance of $\Xi_t$, as a function of $t$, that we denote by $\hat{\mathbf \Sigma}(t)$, is
	\begin{eqnarray}\label{eq:densityflow}
		 &&M(t-i,0)\Phi(0,t-i)^T\hat\Sigma_i^{-1/2}\left[-\hat\Sigma_i^{1/2}\Phi^T 
		Q\Phi\hat\Sigma_i^{1/2}+(\hat\Sigma_i^{1/2}\Phi^T Q\hat\Sigma_{i+1}Q\Phi
		\hat\Sigma_i^{1/2})^{1/2}\right.
		\\&&+\left.\hat\Sigma_i^{1/2}\Phi(t-i,0)^T
        M(t-i,0)^{-1}\Phi(t-i,0)\hat\Sigma_i^{1/2}\right]^2 \hat\Sigma_i^{-1/2}\Phi(0,t-i)M(t-i,0)\nonumber
	\end{eqnarray}  
for $i\le t\le i+1$ any $0\le i \le N-1$, see \cite{CheGeoPav17}. Here 
	\[
		\Phi(t,0)= 
		\left[\begin{matrix}
		1 & t\\ 0 & 1
		\end{matrix}\right],
		\quad
		\Phi(0,t)=\Phi(t,0)^{-1},
	\]
and
	\[
		M(t,0) = 
		\left[\begin{matrix}
		t^3/3 & t^2/2\\
		t^2/2 & t
		\end{matrix}\right].
	\]
The covariance $\Sigma_t$ for $X_t$ is the $(1,1)$-block of $\hat{\mathbf \Sigma}(t)$. By combining the interpolations of the means and the covariances, we conclude that the cubic spline interpolation for the $N+1$ Gaussian marginals is a Gaussian density flow with mean $m(t)$ and covariance $\Sigma_t$ for $0\le t\le N$.

\section{Numerical examples}\label{sec:numerical}
In order to illustrate the framework, we concluded with numerical examples of density-curves that interpolate a set of specified Gaussian marginals.
For simplicity we consider the marginals to be $1$-dimensional and have zero-mean, and we focus on how the density-curve interpolates the respective variances.
We generate our initial data (a set of variances) randomly, and then solve \eqref{eq:effSDP} to obtain the variances corresponding to density-curve through \eqref{eq:densityflow}. Figures \ref{fig:N5}, \ref{fig:N10}, and \ref{fig:N100} depict results for different values of $N$. 
It is noted that the one-dimensional curves shown in these plots, which deligneate the values of interpolating-variance as function of $t$, differ from cubic splines on $\R$; cubic splines would not preserve positivity in general whereas the construction in (\ref{eq:effSDP}-\ref{eq:densityflow}) obviously does.
\begin{figure}[h]
\includegraphics[width=0.6\textwidth]{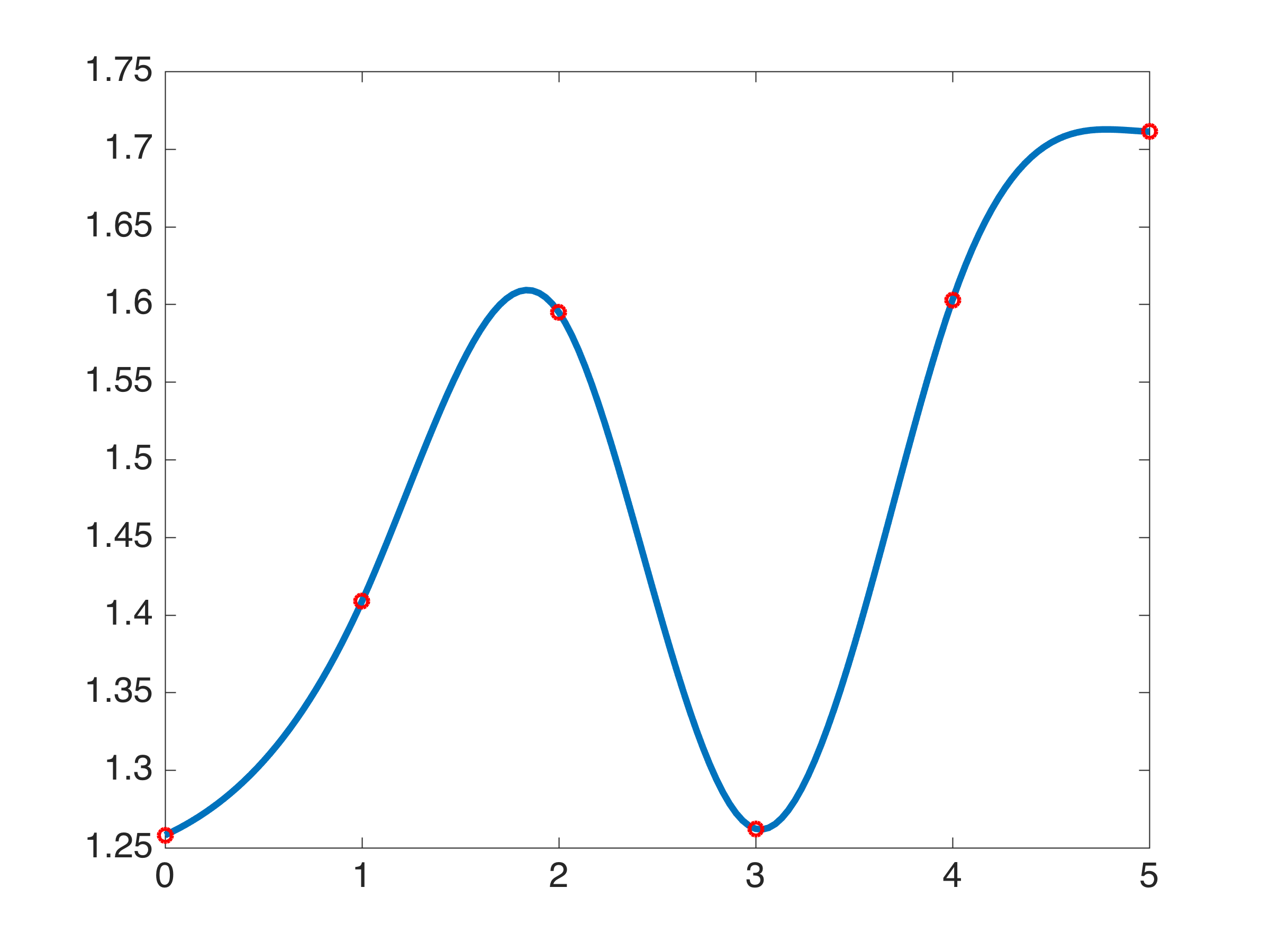}
\caption{Interpolation of covariances: $N=5$}
\label{fig:N5}
\end{figure}
\begin{figure}[h]
\includegraphics[width=0.6\textwidth]{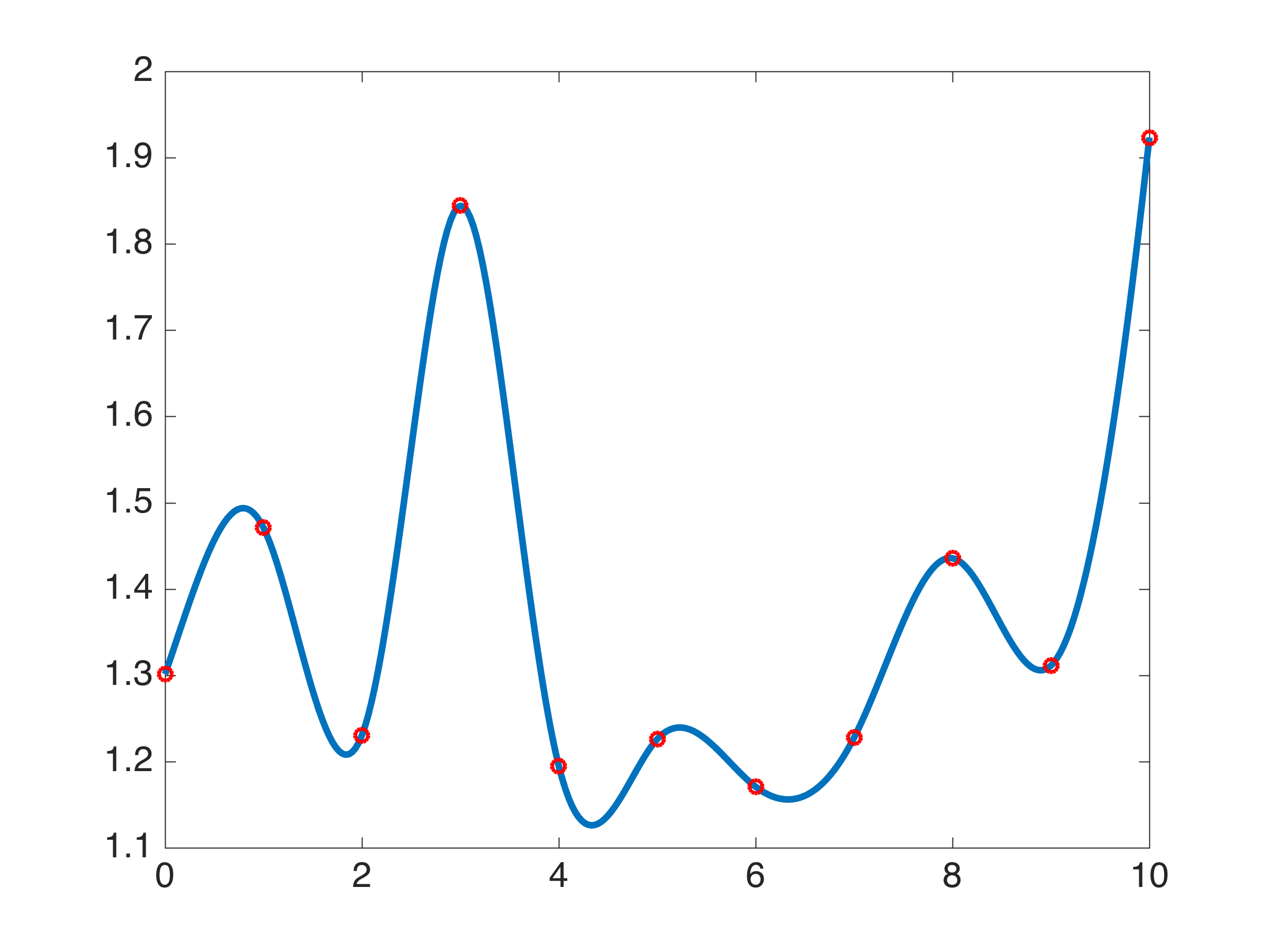}
\caption{Interpolation of covariances: $N=10$}
\label{fig:N10}
\end{figure}
\begin{figure}[h]
\includegraphics[width=0.6\textwidth]{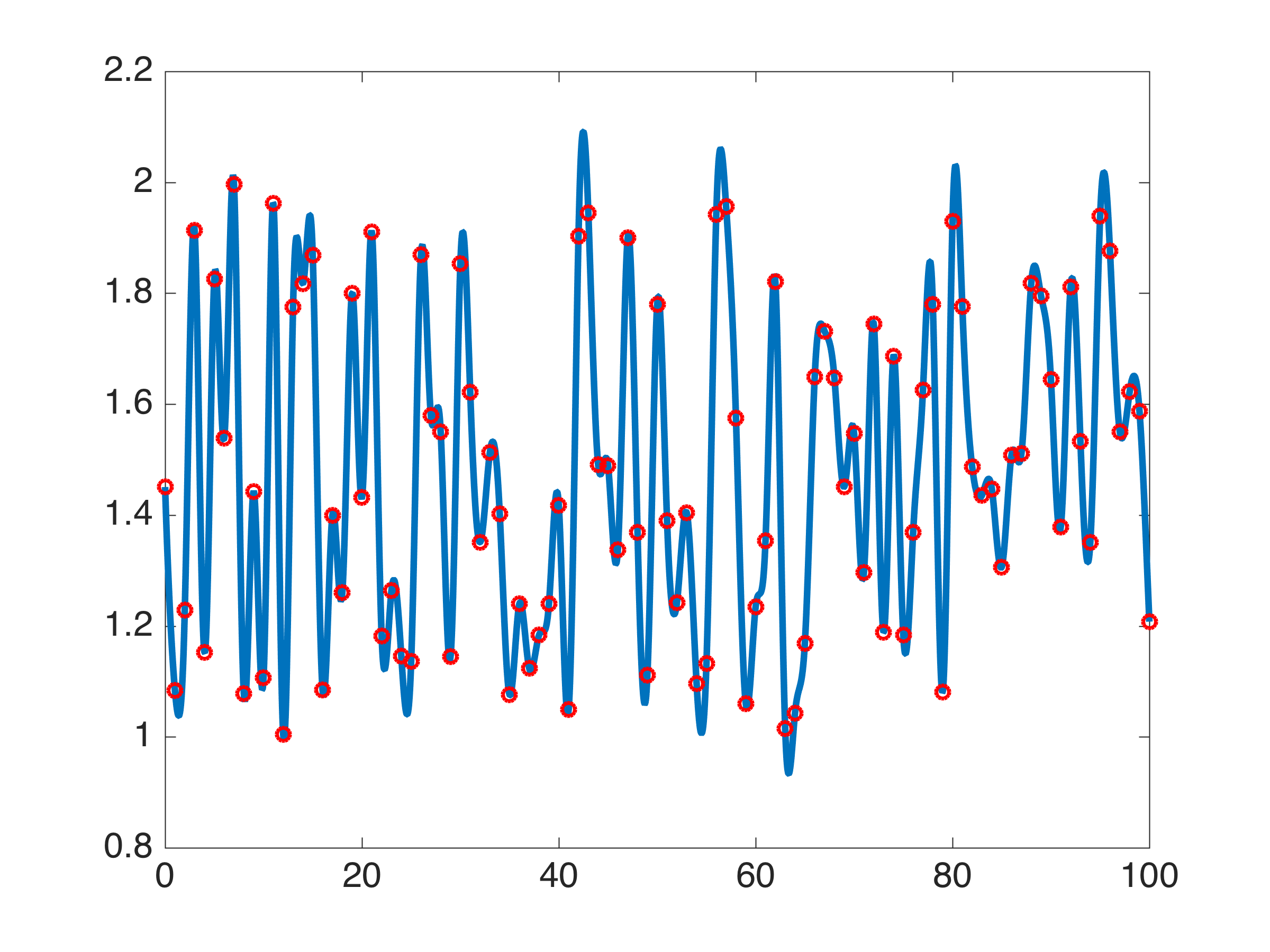}
\caption{Interpolation of covariances: $N=100$}
\label{fig:N100}
\end{figure}


\bibliographystyle{plain}
\bibliography{Ref}

\end{document}